\documentclass[12pt, reqno]{amsart}
\usepackage{amssymb, mathrsfs}
\usepackage{latexsym}
\usepackage{amsmath}
\usepackage{amsthm}
\usepackage{amsfonts}
\usepackage{dsfont, upgreek}
\usepackage{mathtools}
\usepackage{epsfig}
\usepackage[nocompress]{cite}
\usepackage{amscd}
\usepackage{graphicx}
\usepackage{tikz-cd}
\usepackage{enumitem}
\setlist[enumerate]{itemsep=0.5ex}

\usepackage{color}
\usepackage{tikz}
\usetikzlibrary{patterns}

\usepackage[left=1.4in,right=1.4in,top=1.2in,bottom=1.2in]{geometry}

\usepackage{tikz}
\usetikzlibrary{decorations.markings}
\usetikzlibrary{arrows.meta}

\usepackage{extarrows}

\usepackage{adjustbox}
\usepackage{pgfplots}
\usepgfplotslibrary{colormaps}
\usepackage{slashed}

\usepackage[colorlinks=true,linkcolor=blue,citecolor=blue]{hyperref}

\usepackage[colorinlistoftodos,prependcaption,textsize=tiny]{todonotes}  

\usepackage[all,cmtip]{xy}
\theoremstyle{plain}
\newtheorem{theorem}{Theorem}[section]
\newtheorem{proposition}[theorem]{Proposition}
\newtheorem{lemma}[theorem]{Lemma}
\newtheorem{corollary}[theorem]{Corollary}
\newtheorem{question}[theorem]{Question}

\theoremstyle{definition} 
\newtheorem{definition}[theorem]{Definition}

\newtheorem*{claim*}{Claim}

\theoremstyle{remark} 
\newtheorem{remark}[theorem]{Remark}

\numberwithin{equation}{section}

\newcommand{\Sc}{\mathrm{Sc}}

\newcommand{\Bigwedge}{\mathord{\adjustbox{raise=.4ex, totalheight=.7\baselineskip}{$\bigwedge$}}}

\newcommand{\ind}{\textup{Ind}}
\newcommand{\id}{\mathrm{id}}

\newcommand{\R}{\mathbb{R}}
\newcommand{\tr}{\mathrm{tr}}
\newcommand{\sph}{\mathbb{S}}
\newcommand{\equator}{\mathbb E}
\newcommand{\T}{\mathbb{T}}
\newcommand{\dist}{\textup{dist}}
\newcommand{\SO}{\mathrm{SO}}
\newcommand{\codim}{\mathrm{codim}}

\newcommand{\interior}[1]{%
	{\kern0pt#1}^{\mathrm{\,o}}%
}

\makeatletter
\let\save@mathaccent\mathaccent
\newcommand*\if@single[3]{%
	\setbox0\hbox{${\mathaccent"0362{#1}}^H$}%
	\setbox2\hbox{${\mathaccent"0362{\kern0pt#1}}^H$}%
	\ifdim\ht0=\ht2 #3\else #2\fi
}
\newcommand*\rel@kern[1]{\kern#1\dimexpr\macc@kerna}
\newcommand*\overbar[1]{\@ifnextchar^{{\wide@bar{#1}{0}}}{\wide@bar{#1}{1}}}
\newcommand*\wide@bar[2]{\if@single{#1}{\wide@bar@{#1}{#2}{1}}{\wide@bar@{#1}{#2}{2}}}
\newcommand*\wide@bar@[3]{%
	\begingroup
	\def\mathaccent##1##2{%
		\let\mathaccent\save@mathaccent
		\if#32 \let\macc@nucleus\first@char \fi
		\setbox\z@\hbox{$\macc@style{\macc@nucleus}_{}$}%
		\setbox\tw@\hbox{$\macc@style{\macc@nucleus}{}_{}$}%
		\dimen@\wd\tw@
		\advance\dimen@-\wd\z@
		\divide\dimen@ 3
		\@tempdima\wd\tw@
		\advance\@tempdima-\scriptspace
		\divide\@tempdima 10
		\advance\dimen@-\@tempdima
		\ifdim\dimen@>\z@ \dimen@0pt\fi
		\rel@kern{0.6}\kern-\dimen@
		\if#31
		\overline{\rel@kern{-0.6}\kern\dimen@\macc@nucleus\rel@kern{0.4}\kern\dimen@}%
		\advance\dimen@0.4\dimexpr\macc@kerna
		\let\final@kern#2%
		\ifdim\dimen@<\z@ \let\final@kern1\fi
		\if\final@kern1 \kern-\dimen@\fi
		\else
		\overline{\rel@kern{-0.6}\kern\dimen@#1}%
		\fi
	}%
	\macc@depth\@ne
	\let\math@bgroup\@empty \let\math@egroup\macc@set@skewchar
	\mathsurround\z@ \frozen@everymath{\mathgroup\macc@group\relax}%
	\macc@set@skewchar\relax
	\let\mathaccentV\macc@nested@a
	\if#31
	\macc@nested@a\relax111{#1}%
	\else
	\def\gobble@till@marker##1\endmarker{}%
	\futurelet\first@char\gobble@till@marker#1\endmarker
	\ifcat\noexpand\first@char A\else
	\def\first@char{}%
	\fi
	\macc@nested@a\relax111{\first@char}%
	\fi
	\endgroup
}
\makeatother

\begin{document}
	\title[Scalar curvature rigidity and $L^\infty$ metrics ]{Scalar curvature rigidity of spheres with subsets removed and $L^\infty$ metrics}
	
	\author{Jinmin Wang}
	\address[Jinmin Wang]{Institute of Mathematics, Chinese academy of sciences}
	\email{jinmin@amss.ac.cn}
	\thanks{}
	\author{Zhizhang Xie}
	\address[Zhizhang Xie]{ Department of Mathematics, Texas A\&M University }
	\email{xie@tamu.edu}
	\thanks{The second author is partially supported by NSF  1952693 and 2247322.}
	 
	\begin{abstract}
We prove the scalar curvature rigidity for $L^\infty$ metrics on $\mathbb S^n\backslash\Sigma$, where $\mathbb S^n$ is the $n$-dimensional sphere with $n\geq 3$ and $\Sigma$ is a closed subset of $\mathbb S^n$ of codimension at least $\frac{n}{2}+1$ that  satisfies the wrapping property. The notion of wrapping property was introduced by the second author for studying related scalar curvature rigidity problems on spheres. For example,  any closed subset of $\mathbb S^n$ contained in a hemisphere and any finite subset of $\mathbb S^n$ satisfy the wrapping property.  The same techniques also apply to prove an analogous scalar rigidity result for $L^\infty$ metrics on tori that are smooth away from certain subsets of codimension at  least $\frac{n}{2}+1$. As a corollary, we obtain a positive mass theorem for complete asymptotically flat spin manifolds with arbitrary ends for $L^\infty$ metrics. 
	\end{abstract}
	\maketitle

\section{introduction}

The scalar curvature rigidity problems play a central role in Riemannian geometry. A well known theorem of Llarull \cite{Llarull} states that if $g$ is a smooth Riemannian metric on $\sph^n$ such that $g\geq g_{\sph^n}$ and $\Sc_{g}\geq \Sc_{g_{\sph^n}} = n(n-1)$, then $g= g_{\sph^n}$, where $g_{\sph^n}$ is the standard round metric on $\sph^n$. More generally, Llarull showed that if $f\colon (M, g) \to (\sph^n, g_{\sph^n})$ is a nonzero degree area-non-increasing  map such that $\Sc_g \geq \Sc_{g_{\sph^n}}$, then $f$ is an isometry. This was later generalized by Goette and Semmelmann \cite{GoetteSemmelmann} to the scalar curvature rigidity of closed manifolds with nonnegative curvature operators and nonzero Euler characteristics. In more recent years, the theorems of Llarull and Goette-Semmelmann have been further generalized to manifolds with smooth boundary \cite{Lottboundary} and manifolds with corners (more generally manifolds with polyhedral boundaries) \cite{Wang:2021tq, Wang:2022vf}. 

In \cite[Section 10.9]{Gromovinequalities2018}, Gromov raised the following open question regarding the scalar curvature rigidity for spheres with some subset removed. 

\begin{question}
    Let $\Sigma$ be a closed subset of the standard round sphere $(\sph^n, g_{\sph^n})$.
Suppose $(X, g)$  is an orientable $n$-dimensional Riemannian manifold and 
\[ f\colon (X, g) \to (\sph^n\backslash \Sigma, g_{\sph^n}) \]
is a nonzero degree smooth proper map such that $f$ is area-non-increasing and 
\[   \Sc_g \geq \Sc_{g_{\sph^n}} = n(n-1).\]
Does it follow that $f$ is an isometry?
\end{question}

In the case where $\Sigma$ consists of two antipodal points,  if $\dim X = 3$, the above question was positively confirmed by Hu-Liu-Shi \cite{MR4594778} via Gromov's $\mu$-bubble approach and  Hirsch-Kazaras-Khuri-Zhang \cite{Hirsch:2022aa} via harmonic functions. For general dimensions, B\"ar-Brendle--Hanke--Wang \cite{Baer:2023aa} and the authors \cite{WangXie23} independently confirmed the above question for spin manifolds in all dimensions $(\geq 3)$ in the case of  $\Sigma$ consists of two antipodal points, via Dirac operator techniques.  More recently, Chu, Lee and Zhu answered positively the above question for spin manifolds in the case where $\Sigma$ is a finite set of points, under the additional assumption that $g$ is a $L^\infty$ metric \cite{ChuLeeZhu24}. More precisely, they proved the following theorem: if $M^n$ ($n\geq 3$) is closed spin manifold and $S$ is finite subset of $M$ so that $g$ is an $L^\infty$-metric on $M$ that is smooth outside $S$ with $\Sc_g \geq n(n-1)$ on $M\backslash S$, then any nonzero degree distance-non-increasing Lipschitz map $f\colon (M,g)\to (\sph^n,g_{\sph^n})$ that is  smooth away from $S$ has to be an isometry. 

A main purpose of the present paper is to answer positively the above question of Gromov for a wide class of subsets $\Sigma$, under the additional $L^\infty$ assumption on the metric $g$.
To be more precise, let us first recall the definition of $L^\infty$ Riemannian metrics.

\begin{definition}
	Let $M$ be a closed smooth manifold. Let $g_0$ be a smooth Riemannian metric on $M$. We say that a measurable section $g$ over $M$ with value in positive definite  symmetric $2$-tensor is an $L^\infty$-metric, if there exists $C>0$ such that
	$C^{-1}g_0\leq g\leq Cg_0$ holds almost everywhere.
\end{definition}

Here is the first main theorem of the paper.
\begin{theorem}\label{thm:sphere}
	Let $\Sigma$ be a finite simplicial complex embedded in $\sph^n$ ($n\geq 3$) of codimension $\geq 3$.   Let $M^n$  be a closed spin manifold and $g$  an $L^\infty$-metric on $M$. Let $f\colon (M,g)\to (\sph^n,g_{\sph^n})$ be a Lipschitz map with non-zero degree. Assume that  
	\begin{enumerate}[label=$(\arabic*)$]
	\item $S = f^{-1}(\Sigma)$ is a finite simplicial complex embedded in $M$ of codimension $k \geq\frac n 2+1$,
	\item the Riemannian metric $g$ on $M$ is smooth away from $S$,
	\item the map $f$ is smooth away from $S$, and 
	\item  $\Sigma\subset\sph^n$ has the  wrapping property.	
	\end{enumerate}
If we have \begin{equation}\label{eq:scalarcomparison}
	\Sc_{g}\geq |\Bigwedge^2df|\cdot \Sc_{g_{\sph^n}}=|\Bigwedge^2df|\cdot n(n-1)
\end{equation} on $M\backslash S$,  then $f$ is a metric isometry on $M$ up to scaling.\footnote{This means that there exists a constant $c>0$ such that $d_{g_{\sph^n}}(f(x), f(y)) = c\cdot d_g(x, y)$ for all $x, y\in M$. In particular, $f$ is a Riemannian isometry on $M\backslash S$ after scaling. } Here $|\Bigwedge^2df_x|$ is the norm of the map 
\[ \Bigwedge^2df_x\colon \Bigwedge^2 T_xM \to \Bigwedge^2 T_{f(x)}\sph^n.\] 
\end{theorem}
To be precise,  we say $S$ is a finite simplicial complex embedded in a Riemannian manifold $M$ if  there exists an injective map $\psi\colon S \to M$ such that $\psi\colon S\to \psi(S)$ is a bi-Lipschitz map, where $S$ is equipped with the standard simplicial metric and $\psi(S)$ is equipped with the metric inherited from $M$. If no confusion is likely to arise, we shall write $S$ in place of its image $\psi(S)$ in $M$ throughout the paper.   Let us recall  the notion of wrapping property, which  was introduced the second author for studying related scalar curvature rigidity results on spheres \cite{Xie:2021tm}.   \begin{definition}\label{def:wrapping}
	Let $\Sigma$ be a closed subset of $\sph^n$. We say that $\Sigma$ satisfies the wrapping property if there exists a $1$-Lipschitz map $h\colon \sph^n\to\sph^n$ such that
	\begin{enumerate}
		\item $\deg(h)\ne 1$, and
		\item when restricted to any connected component of $\Sigma$, $h$ is equal to an isometry in $\SO(n+1)$.
	\end{enumerate}
\end{definition}

There are many examples of closed subsets of spheres that satisfy the wrapping property, such as any finite subset or a closed subset contained in an open hemisphere. In particular, we have  the following corollary of Theorem \ref{thm:sphere}.
\begin{corollary}\label{coro:point&tree}
	
	Let $\Sigma$ be a finite simplicial complex embedded in $\sph^n$ ($n\geq 3$) of codimension $\geq 3$.   Let $M^n$  be a closed spin manifold and $g$  an $L^\infty$-metric on $M$. Let $f\colon (M,g)\to (\sph^n,g_{\sph^n})$ be a Lipschitz map with non-zero degree. Assume that  
	\begin{enumerate}[label=$(\arabic*)$]
		\item $S = f^{-1}(\Sigma)$ is a finite simplicial complex embedded in $M$ of codimension $k \geq\frac n 2+1$,
		\item the Riemannian metric $g$ on $M$ is smooth away from $S$,
		\item the map $f$ is smooth away from $S$.	
	\end{enumerate}
		Suppose that
	\begin{equation*}
\Sc_{g}\geq |\Bigwedge^2df|\cdot \Sc_{g_{\sph^n}}=|\Bigwedge^2df|\cdot n(n-1)
	\end{equation*}  and one of the following holds
		\begin{enumerate}[label=$(\alph*)$]
			\item $\Sigma$ is contained in a hemi-sphere,
			\item $\Sigma$ is finite,
			\item or $\Sigma$ is a disjoint union of trees,
		\end{enumerate}
		then $f$ is a metric isometry on $M$ up to scaling. 
\end{corollary}

As mentioned above,  in the special case where $S$ is a finite set of points, Theorem $\ref{thm:sphere}$ was recently obtained by Chu-Lee-Zhu \cite{ChuLeeZhu24}. The statement of their theorem \cite[Theorem 1.4]{ChuLeeZhu24} requires a slightly stronger assumption that $f$ is area nonincreasing and $\Sc_g \geq \Sc_{g_{\sph^n}} = n(n-1)$, but it is not difficult to see that a slight modification of their proof actually shows the rigidity under the weaker assumption \eqref{eq:scalarcomparison}. In fact, our strategy for proving Theorem \ref{thm:sphere} is largely inspired by the work of Chu-Lee-Zhu \cite{ChuLeeZhu24}, combined with some earlier work of the second author on related scalar curvature rigidity results \cite{Xie:2021tm}.

\begin{remark}\label{example:L-inf}
	The following example shows that the $L^\infty$ assumption on the metric $g$ is  necessary for Theorem \ref{thm:sphere}. Let us write the standard metric on $\mathbb S^n$ as 
	$$g_{\mathbb S^n}=dr^2+\cos^2(r)g_{\mathbb S^{n-1}}$$ for $r\in [-\frac{\pi}{2},\frac{\pi}{2}]$. Let us consider a different warped product metric given by
	$$g=\frac{1}{n^2}dr^2+n^4\cos^2(\frac r n)g_{\mathbb S^{n-1}},$$ which is smoothly defined on $\mathbb S^n$ away from the poles. By specifying an inner product at the two poles, we can view $g$ as a singular metric on all of $\mathbb S^n$. 
	Note that $g$ is clearly not an $L^\infty$-metric on $\mathbb S^n$. Let 
	$$f\colon (\mathbb S^n,g)\to (\mathbb S^n,g_{\mathbb S^n})$$
	 be the identity map. Then
	$$|\Bigwedge^2 df|=\max\left\{\frac{\cos^2(r)}{n^4\cos^2(\frac r n)},n\cdot \frac{\cos(r)}{n^2\cos(\frac r n)}\right\}\leq \frac{n^3}{n^4\cos^2(\frac r n)}.$$
	Set $t=\frac r n$ and $\varphi(t)=n^2\cos(t)$. Then
	\begin{align*}
		\mathrm{Sc}_g=&\frac{(n-1)(n-2)}{\varphi^2}-\frac{(n-1)(n-2)\varphi'^2}{\varphi^2}-2(n-1)\frac{\varphi''}{\varphi}\\
		=&(n-1)\frac{n-2-n^5\sin^2(\frac r n)+2n^4}{n^4\cos^2(\frac r n)}
	\end{align*}
	However,
	\begin{align*}
		\mathrm{Sc}_g-|\Bigwedge^2df|\cdot n(n-1)
		&\geq(n-1)\frac{n-2-n^5\sin^2(\frac r n)+n^4}{n^4\cos^2(\frac r n)}\\
		\geq&
		\frac{n-1}{n^4\cos^2(\frac r n)}\left(n-2+n^3(n-\frac{\pi^2}{4})\right),
	\end{align*}
	which is positive if $n\geq 3>\frac{\pi^2}{4}$. This gives an example that would  violate Theorem \ref{thm:sphere}, if we were to remove the $L^\infty$ assumption on  $g$. 
	

We should remark that, if we replace  the assumption \eqref{eq:scalarcomparison} by the stronger assumption that $f$ is distance-non-increasing and $\Sc_{g}\geq  n(n-1)$, then  the rigidity statement in Theorem \ref{thm:sphere} holds for $\sph^n\backslash\{\pm\}$ without the $L^\infty$ assumption on $g$, where $\{\pm\}$ stands for a pair of antipodal points. In particular, if  $g$ is a (possibly not $L^\infty$) Riemannian metric  $\sph^n$ that is smooth away from $\{\pm\}$  such that $g\geq g_{\sph^n}$ and $\Sc_g\geq n(n-1)$, then $g = g_{\sph^n}$. See \cite{Hirsch:2022aa, MR4594778} for $n=3$ and \cite{Baer:2023aa,WangXie23} for arbitrary dimensions.
For more general subsets $\Sigma$, it remains an open question whether such a scalar curvature rigidity result holds under the assumption that $f$ is distance-non-increasing and $\Sc_{g}\geq  n(n-1)$, but without  the $L^\infty$ assumption. 
\end{remark}
\begin{remark}\label{example:>=3}
	The following example shows  that the codimension of $S$ has to be at least $ 3$. Consider the standarnd round metric $g_{\sph^2}$ on $\sph^2$, which may be viewed as  $dr^2+(\sin r)^2d\theta^2$ for $r\in[0,\pi]$ and $\theta\in[0,2\pi]$. Let $\{\pm\}$ be the two poles at $r=0,\pi$ in $\sph^2$, which has codimension $2$. Consider the metric 
	$$\overbar g_{\sph^2}\coloneqq dr^2+C^2\sin^2 rd\theta^2.$$
	on $\sph^2\backslash\{\pm\}$ with $C>1$. Clearly, we have  $g_{\sph^2}\leq \overbar g_{\sph^2}\leq C^2g_{\sph^2}$, hence $\overbar g_{\sph^2}$ is an $L^\infty$-metric on $\sph^2$ that is smooth away from $\{\pm\}$. The identity map $f\colon (\sph^2,\overbar g_{\sph^2})\to(\sph^2,g_{\sph^2})$ satisfies
	$$|\Bigwedge^2df(r,\theta)|=C^{-1}<1,$$
	but we have $\Sc_{\overbar g_{\sph^2}}= \Sc_{g_{\sph^2}}>C^{-1}\Sc_{g_{\sph^2}}$.
	
	By induction, we may construct a similar  example in any higher dimension. Assume that we have already constructed a $L^\infty$ metric $\overbar g_{\sph^n}$ on $\sph^n$ that is smooth away from some codimension $2$ subset $\Sigma_n$ such that $g_{\sph^n}\leq\overbar g_{\sph^n}\leq C^2 g_{\sph^n}$ and $\Sc_{\overbar g_{\sph^n}}\geq \Sc_{g_{\sph^n}}$. Then the metric
	$$\overbar g_{\sph^{n+1}}\coloneqq dr^2+(\sin r)^2 \overbar g_{\sph^n}$$
	is a $L^\infty$ metric $\overbar g_{\sph^n}$ on $\sph^n$ that is smooth away from a  subset $\Sigma_{n+1}$ such that $g_{\sph^n}\leq\overbar g_{\sph^n}\leq C^2 g_{\sph^n}$ and $\Sc_{\overbar g_{\sph^n}}\geq \Sc_{g_{\sph^n}}$. Here $\Sigma_{n+1}$ consists of  all geodesics between the two poles that pass through $\Sigma_n$. In particular, $\Sigma_{n+1}$ has codimension $2$ in $\sph^{n+1}$. 
 
 We prove  Theorem \ref{thm:sphere} under the condition that $\codim S\geq \frac{n}{2}+1$, which implies $\codim S\geq 3$, since $n\geq 3$. In general, it remains an open question whether Theorem \ref{thm:sphere} still hold under the weaker assumption that  $\codim S\geq 3$. 
\end{remark}

\begin{remark}
	Under the assumptions of Theorem \ref{thm:sphere}, although we conclude that $f\colon M\to\sph^n$ is a metric isometry, it is \emph{not} smoothable  on the entire $M$ in general. For example, let us consider $(M,g_0)$ to be a smooth exotic sphere for $n=7$, where $g_0$ is a Riemannian metric on $M$ that is smooth with respect to the exotic smooth structure, cf. \cite{Cecchini:2024aa}. The smooth structure on $M$ can be obtained by gluing two copies of the $n$-dimensional ball $\mathbb B^n$ along their  boundary $\sph^{n-1}$ via some nontrivial diffeomorphism  $\varphi\colon \sph^{n-1}\to \sph^{n-1}$.  Let $x_0$ be the center of $\mathbb B^n$. Then the map  $\varphi\colon \sph^{n-1}\to \sph^{n-1}$ extends conically to a continuous map $\Phi\colon \mathbb B^n \to \mathbb B^n$. Note that $\Phi$ is smooth away from $x_0$, but not smooth at $x_0$. We call the images of the two copies of $\mathbb B^n$ in $M$ upper and lower hemispheres.  Consider the following map $f\colon M\to \sph^n$: $f$ is the identity map on the lower hemisphere, and $f=\Phi$ on the upper hemisphere. It is not difficult to see that  $g\coloneqq  f^*g_{\sph^n}$ is  $L^\infty$-metric on $M$ (with respect to the background metric $g_0$). Now $f\colon (M,g)\to (\sph^n,g_{\sph^n})$ satisfies all the conditions in Theorem \ref{thm:sphere}, and clearly $f$ is a metric isometry by construction. But $f$ is not smoothable on $M$.  
\end{remark}

The same techniques used to prove Theorem \ref{thm:sphere} also apply to give a partial answer to the following open question of Schoen (cf. \cite[Conjecture 1.5]{MR3961306}). 
\begin{question}\label{ques:schoen}
	Suppose $g$ is a $L^\infty$ metric on the torus $\mathbb T^n$ that is smooth away from a closed embedded submanifold $S$ of codimension $\geq 3$. If $\Sc_g\geq 0$ on $\mathbb T^n-S$, then does  $g|_{\mathbb T^n-S}$ extends to a smooth flat metric on $\mathbb T^n$?  
\end{question}

More precisely, we have the following theorem on the scalar curvature rigidity for tori with certain subsets removed.  

\begin{theorem}\label{thm:torus-intro}
	Let $M^n$ ($n\geq 3$) be a closed spin manifold and $S$ a finite simplicial complex embedded in $M$ of codimension $k$. Let $g$ be an $L^\infty$-metric on $M$ that is smooth outside $S$. Let $f\colon (M,g)\to (\T^n,g_{\T^n})$ be a Lipschitz map with non-zero degree. Suppose that $\Sc_{g}\geq 0$ on $M\backslash S$. If $k = \codim S  \geq\frac n 2+1$ and $$f_*\colon \pi_1(S)\to \pi_1(\T^n)$$ is the zero map, 
	then $g$ is Ricci flat on $M\backslash S$.
\end{theorem}

We remark that the above theorem (and the same proof) also hold if  we replace the torus $\mathbb T^n$  by any closed enlargeable spin manifold (cf.  \cite[Definition 6.4]{GromovLawson}). In the special case where $S$ consists of finitely points, Theorem \ref{thm:torus-intro} also follows from Theorem C of a recent preprint of Cecchini-Frenck-Zeidler  \cite{Cecchini:2024aa}.  We would like to also mention that there is a more general version of Question \ref{ques:schoen} for closed manifolds with nonnegative smooth $\sigma$-invariant (cf. \cite[Conjecture 1.5]{MR3961306}). In the same preprint \cite{Cecchini:2024aa},  Cecchini, Frenck and Zeidler provided a negative answer to this more general version of Question \ref{ques:schoen}. 

As a consequence of Theorem \ref{thm:torus-intro}, we have the following version of positive mass theorem for $L^\infty$ metrics. See \cite{DaiSunWang24} for a  positive mass theorem for metrics with isolated conical singularities. 

\begin{theorem}
	Let $M^n$ be a smooth manifold  equipped with a $L^\infty$ Riemannian metric $g$ such that $(M^n, g)$ is a complete asymptotically flat spin manifold with arbitrary ends.\footnote{See Definition \ref{def:af} for the precise definition of complete asymptotically flat spin manifolds with arbitrary ends.} Suppose  $S$ a finite simplicial complex embedded in $M$ of codimension $\geq \frac{n}{2} +1$ and  $g$ is smooth outside $S$. If $\Sc_g\geq 0$, then the ADM mass $m_g$ of $(M, g)$ is greater than or equal to $ 0$. Furthermore, if  $m_g=0$, then $g$ is Ricci flat.
\end{theorem}

The paper is organized as follows. In Section \ref{sec:sphererigid}, we prove Theorem \ref{thm:sphere}, a scalar curvature rigidity result of $\sph^n\backslash \Sigma$ for $L^\infty$ metrics, where $\Sigma$ satisfies the wrapping property. In Section \ref{sec:tori}, we prove Theorem \ref{thm:torus-intro}, a scalar curvature rigidity result for tori with certain subsets removed. In Section \ref{sec:pmt}, we apply Theorem \ref{thm:torus-intro} to prove a positive mass theorem for complete $L^\infty$ metrics.
In the Appendix \ref{app:wrapping}, we show that each finite subset of the sphere $\sph^n$ satisfies the wrapping property.  

We would like to thank Simone Cecchini for helpful comments.

\section{Scalar curvature rigidity for $\sph^n\backslash \Sigma$}\label{sec:sphererigid}
In this section, we prove Theorem \ref{thm:sphere}, a scalar curvature rigidity result of $\sph^n\backslash \Sigma$ for $L^\infty$ metrics, where $\Sigma$ satisfies the wrapping property. 

\subsection{Green function and conformal deformation}\label{sec:Green}
Let us retain the same notation from Theorem \ref{thm:sphere}. In this subsection, following Chu-Lee-Zhu \cite{ChuLeeZhu24},  we use Green functions to construct a complete metric $g_\times$ on $M\backslash S$ while maintaining the assumptions of  Theorem \ref{thm:sphere} (in particular, inequality \eqref{eq:scalarcomparison}) for the new metric $g_\times$.

We first define a conformal Laplacian on $M$ associated to the map $f$. For any point $x\in M\backslash S$, let $\Bigwedge^2df_x\colon \Bigwedge^2T_xM\to \Bigwedge^2T_{f(x)}\sph^n$ be the linear map on $2$-forms induced by $df$. Define
$$\sigma_f\colon M\backslash S\to \R_{\geq 0},\quad x\mapsto 2|\Bigwedge^2df_x|_{\tr},$$
where $|\cdot|_{\tr}$ denote the trace norm
Since
$$|\Bigwedge^2df_x|_{\tr}=\sqrt{\tr\left(\left(\Bigwedge^2df_x\right)^*\circ\Bigwedge^2df_x\right),}$$
$\sigma_f(x)$ is not smooth only if $\sigma_f(x)=0$. We define
\begin{equation}\label{eq:laplace}
	\mathcal L=-\frac{4(n-1)}{n-2}\Delta_g+\phi(\Sc_g-\sigma_f)
\end{equation}
where $\phi\colon \R_{\geq 0}\to\R_{\geq 0}$ is a smooth non-decreasing bounded function\footnote{The specific form of the function $\phi$ is not so important for the construction that follows, as long as the listed properties are satisfied.} such that
\begin{itemize}
	\item $\phi(s)\leq s,~\forall s\geq 0$,
	\item $\phi(s)=0$ only if $s=0$,
	\item $\phi(s)\leq 20$, and $\frac{d^k\phi}{dt^k}(0)=0$ for any $k\geq 0$.
\end{itemize}
Here $-\Delta_g$ is the Laplace-Beltrami operator on $M$ (with respect to the $L^\infty$ metric $g$). We refer the reader to \cite{MR0161019} and \cite{MR0657523} for more details on some key properties of Laplace-Beltrami operators when the background metrics are only $L^\infty$ metrics. To avoid any possible confusion, our sign convention is that $-\Delta_g$ is a nonnegative operator. 

By construction, if $\Sc_g\geq\sigma_f$, then $\phi(\Sc_g-\sigma_f)$ is a smooth non-negative function on $M\backslash S$, and $\mathcal L$ is non-negative. Furthermore, in this case, if the spectrum of $\mathcal L$ contains zero, then its bottom eigen-function is a nonzero constant function. It follows that $\phi(\Sc_g-\sigma_f)\equiv 0$ on $M\backslash S$, hence $\Sc_g\equiv \sigma_f$ on $M\backslash S$. Thus if $\Sc_g\geq\sigma_f$ and $\Sc_g>\sigma_f$ somewhere on $M\backslash S$, then $\mathcal L$ is strictly positive.

Recall that $f\colon M\to\sph^n$ is called a metric isometry if
$$d_{g_{\sph^n}}(f(x),f(y))=d_g(x,y)$$
for any $x,y\in M$.
The following lemma shows that if the map $f\colon M \to \sph^n$ is not a metric isometry up to scaling,  then we may construct a new $L^\infty$ metric $g_1$ on $M$ such that $g_1$ is smooth away from $S$ and the inequality \eqref{eq:scalarcomparison} becomes  a strict inequality for all $y\in M\backslash S$.

\begin{lemma}\label{lemma:bound}
	With the same assumptions as Theorem \ref{thm:sphere}, if $f$ is not a metric isometry up to scaling,  then there exist $\delta_1>0$ and an $L^\infty$-metric $g_1$ such that $g_1$ is smooth away from $S$ and $f_1\coloneqq f\colon (M,g_1)\to (\sph^n,g_{\sph^n})$ satisfies $\Sc_{g_1}\geq \sigma_{f_1}+\delta_1$ on $M\backslash S$, where $\sigma_{f_1}(y) = 2|\Bigwedge^2d(f_1)_y|_{\tr}$ for $y\in M\backslash S$.
\end{lemma}
\begin{proof}
	Let us first prove the following claim.
		\begin{claim*}
		Under the assumption of Theorem \ref{thm:sphere},
		\begin{itemize}
			\item either $f$ is a metric isometry up to scaling, 
			\item or $\Sc_g(x)>\sigma_f(x)$ for some $x\in M\backslash S$.
		\end{itemize}
	\end{claim*} 
	By assumption,  we have
	$$\Sc_g(x)\geq |\Bigwedge^2df_x|\cdot n(n-1)$$
	for all  $x\in M\backslash S$.
	It is clear from the H\"{o}lder inequality that
	$$|\Bigwedge^2df_x|\cdot n(n-1)\geq 2|\Bigwedge^2df_x|_{\tr}=\sigma_f(x),$$
	and  the equality holds if and only if $df_x$ is a homethety on $M\backslash S$, namely 
	$$\|df_x(\xi)\|=h\cdot \|\xi\|,~\forall \xi\in T_xM.$$
	Therefore, if we assume that $\Sc_g\equiv \sigma_f$ on $M\backslash S$, then there exists a function $h$ on $M\backslash S$ such that
	$$\|df_x(\xi)\|=h(x)\cdot \|\xi\|,~\forall x\in M\backslash S,\textup{ and }\xi\in T_yM.$$
	In particular, $h$ is uniformly bounded, non-negative, and smooth whenever $h\ne 0$. 
	
	Since $f$ has non-zero degree, there exist points where $h\neq 0$. At these points, we have 
	$$\Sc_g(x)=h(x)^2\cdot n(n-1),\textup{ and }f^*g_{\sph^n}=h^2 g_M.$$
	Set $v=h^{\frac{n-2}{2}}$. It follows that
	$$n(n-1)=v^{-\frac{n+2}{n-2}}\left(-\frac{4(n-1)}{(n-2)}\Delta_g v+\Sc_g v\right)=-v^{-\frac{n+2}{n-2}}\frac{4(n-1)}{(n-2)}\Delta_g v+n(n-1).$$
	Thus $\Delta_g v=0$. It follows from the maximal principle that $v$ is never zero on $M\backslash S$. Furthermore, since $S$ has codimension $\geq 3$, it is a removable singularity of a non-negative bounded harmonic function. It follows that $v$ is a positive constant.
	
	Without loss of generality, we assume that $v\equiv 1$, that is, $df\colon T_xM \to T_{f(x)}\sph^n$ is an isometry for every $x\in M\backslash S$. 
	For any two points $x,y\in M\backslash S$, since the codimension of $S$ is $\geq 3$, there exists a sequence of smooth paths $\gamma_n\colon [0,1]\to M\backslash S$ connecting $x$ and $y$ such that
		$$\lim_{n\to\infty}\textup{length}(\gamma_n)=\lim_{n\to\infty}\int_0^1|\gamma'|dt=\dist_M(x,y).$$
		As $f\colon M\backslash S\to \sph^n\backslash\Sigma$ is a local isometry, we have $\textup{length}(f(\gamma_n))=\textup{lenght}(\gamma_n)$. Consequently, 
		$$\dist_{\sph^n}(f(x),f(y))\leq \dist_M(x,y),~\forall x,y\in M\backslash S.$$
		Since $f\colon M\to \sph^n$ is continuous, $f\colon M\to\sph^n$ is $1$-Lipschitz. Since $\deg f\ne 0$, it follows that $f$ is surjective. Note that the preimage $f^{-1}(b)$ for every $b\in \sph^n\backslash \Sigma$ is finite. Indeed, suppose that $f^{-1}(b)$ were infinite and let $y$ be a limit point of $f^{-1}(b)$. Then
		\begin{itemize}
			\item if $y\in S$, then $f(y)=b$ as $f\colon M\to \sph^n$ is $1$-Lipschitz, which leads to a contradiction as $f(y)\in\Sigma$;
			\item if $y\notin S$, then there is a small neighborhood of $y$, which contains infinitely many points in $f^{-1}(b)$, mapping isometrically to a small neighborhood of $b$, which also leads to a contradiction.
	\end{itemize}
	It follows that  $f\colon M\backslash S\to \sph^n\backslash\Sigma$ is a finite Riemannian covering map. Note  the following map
	$$i_*\colon \pi_1(\sph^n\backslash\Sigma)\to\pi_1(\sph^n)=0$$
	is an isomorphism, since any homotopy of closed curves is generically of dimension two, but $\Sigma=f(S)$ has codimension $\geq 3$. Thus $\pi_1(\sph^n\backslash\Sigma)$ is simply connected. Therefore, $f\colon M\backslash S\to \sph^n\backslash\Sigma$ is a Riemannian isometry and a diffeomorphism. Denote by $f^{-1}\colon \sph^n\backslash\Sigma \to M\backslash S$ its inverse map. The same argument yields that
	$$\dist_{M}(f^{-1}(p),f^{-1}(q))\leq \dist_{\sph^n\backslash\Sigma}(p,q)=\dist_{\sph^n}(p,q),~\forall p,q\in  \sph^n\backslash\Sigma.$$
	Therefore
	$$\dist_{\sph^n}(f(x),f(y))=\dist_M(x,y),~\forall x,y\in M\backslash S.$$
	Since $f\colon M\to \sph^n$ is Lipschitz, we see that $f\colon M\to\sph^n$ is a metric isometry. 	This finishes the proof of the claim.

	Now we choose $x\in M\backslash S$ such that  $df\colon T_xM \to T_{f(x)}\sph^n$ is not an isometry. By the discussion at the beginning of the proof, we have
	$$\Sc_g(x)>\sigma_f(x).$$
	Thus $\Sc_g(y)>\sigma_f(y)$ for any $y$ in a small neighborhood of $x$. As a result, the operator $\mathcal L$ from line \eqref{eq:laplace} is a strictly positive operator.
	
	Let  $\lambda>0$ be the bottom spectrum of $\mathcal L$ and $v$ the associated  eigenfunction. Since $g$ is an $L^\infty$ metric, it follows from the standard elliptic PDE theory (cf. \cite[Corollary 8.11 and Theorem 8.22]{MR1814364}) that  there exists $C>0$ such that  $v\in C^\infty_{\mathrm{loc}}(M\backslash S)\cap C^\alpha(M)$ for some $\alpha\in (0, 1)$. Moreover, by the Harnack inequality (cf. \cite[Theorem 8.20]{MR1814364}), there exists $C>0$ such that 
	$$C^{-1}<v<C.$$
	We define $g_1=v^{\frac{4}{n-2}}g$ and $f_1 \coloneqq f\colon M\to \sph^n$. Since 
	$$\mathcal L(v)=-\frac{4(n-1)}{n-2}\Delta_g(v)+\phi(\Sc_g-\sigma_f)v = \lambda v,$$ it follows that the scalar curvature of $g_1$ on $M\backslash S$ satisfies 
	\begin{align*}
		\Sc_{g_1}=&v^{-\frac{n+2}{n-2}}\Big(-\frac{4(n-1)}{n-2}\Delta_gv+\Sc_gv\Big)
		=v^{-\frac{n+2}{n-2}}\Big(\lambda v-\phi(\Sc_g-\sigma_f)v+\Sc_gv\Big)\\
		\geq &v^{-\frac{n+2}{n-2}}\Big(\lambda v-(\Sc_g-\sigma_f)v+\Sc_gv\Big) = v^{-\frac{4}{n-2}}\Big(\lambda +\sigma_f\Big)\\
		=&\lambda v^{-\frac{4}{n-2}}+\sigma_{f_1}\geq \sigma_{f_1}+\lambda C^{\frac{4}{n-2}}.
	\end{align*}
	where we have used the fact $\phi(x)\leq  x$ and $\sigma_{f_1} = v^{-\frac{4}{n-2}} \sigma_f$. 	This finishes the proof.
\end{proof}

Note that for any $\delta>0$, there exist open sets $V_1$ and $V_2$ with $f(S)\subset V_1\subset V_2$ and a smooth map $h\colon \sph^n\to\sph^n$ with degree one, such that $h(V_1)\subset f(S)$, $|dh|\leq 1+\delta$, and $h=\id$ outside $V_2$.  Considering the composition $$f_2\coloneqq h\circ f_1\colon (M,g')\to \sph^n.$$
We note that
$$|\sigma_{f_1}-\sigma_{f_2}|\leq L^2\big((1+\delta)^2-1\big),$$
where $L<\infty$ is the Lipschitz constant of $f_1$. To summarize, we obtain the following conclusion.
\begin{lemma}\label{lemma:constant}
	With the same assumptions as Theorem \ref{thm:sphere}, if $f\colon (M,g)\to(\sph^n,g_{\sph^n})$ is  not an isometry up to scaling, then there exist $\delta_2>0$,  an $L^\infty$-metric $g_2$ on $M$ and a Lipschitz map $f_2\colon (M,g_2)\to\sph^n$ such that 
	\begin{enumerate}[label=$(\alph*)$]
		\item $g_2$ and $f_2$ are smooth away from $S$,
		\item $f_2(U)\subset f_2(S)$ for some open neighborhood $U$ of $S$,
		\item  $\Sc_{g_2}\geq \sigma_{f_2}+\delta_2$ on $M\backslash S$, where $\sigma_{f_2}(y) = 2|\Bigwedge^2d(f_2)_y|_{\tr}$ for $y\in M\backslash S$.
	\end{enumerate}
\end{lemma}

We will later prove Theorem \ref{thm:sphere} by contradiction. For notation simplicity, we omit the subscripts and assume that $f\colon (M,g)\to\sph^n$ satisfies $\Sc_{g}\geq \sigma_{f}+\delta$ on $M\backslash S$, and $f(U)\subset f(S)$ for some open neighborhood $U$ of $S$.

Now we shall apply a conformal change to obtain a complete Riemannian metric on $M\backslash S$ which still satisfies the assumptions of Theorem \ref{thm:sphere}. As a preparation, we need the following lemma regarding Green functions on $M\backslash S$. 

\begin{lemma}\label{lemma:conformaFactor}
	Let $M$ be an $n$-dimensional closed smooth manifold and  $S\subset M$ a finite simplicial complex with codimension $k$. Suppose $g$ is a $L^\infty$ metric on $M$ that is smooth away from $S$. If $k\geq 3$, then there exist a positive smooth function $u$ on $M\backslash S$ and a continuous positive function $\beta\colon S\to [k, n]$ such that $\mathcal L u=0$ on $M\backslash S$
	and 
	\begin{equation}\label{eq:green}
		\frac 1 C d(x,S)^{2-\beta(x)}\leq u(x)\leq Cd(x,S)^{2-\beta(x)}
	\end{equation}
	for some $C>0$. Here $\mathcal L$ is the operator from line \eqref{eq:laplace} and $d(x, S)$ stands for the distance between $x$ and $S$.  
\end{lemma}
\begin{proof}
	For notation simplicity, we will write $u(x)\sim d(x,S)^{2-k}$ for the desired inequality \eqref{eq:green} in the lemma.	
	Let $G\colon M\times M\to \R_+\cup\{\infty\}$ be the Green function of $(M,g)$ with respect to the operator $\mathcal L$. In particular, we have that $G$ is smooth on $(M\backslash S)\times (M\backslash S)$, $\mathcal LG=0$ and
	$$G(x,y)\sim d(x,y)^{2-n}$$
	for any $x,y\in M$ as $x\to y$. See \cite[Section 6]{MR0161019} and \cite[Section 1]{MR0657523}. 	
	
	Let us fix a measure  $\mu$ on $S$ as follows. For each simplex $\sigma$ in $S$, let $\widetilde \sigma$ be the simplex that has the  maximal dimension\footnote{This (local) maximal dimension does not necessarily equal to the maximum dimension of all simplices in $S$. For example, think of the case where $S$ is the union of a one-dimensional edge and a two dimensional triangle intersecting at a vertex.  } among all simplices of $S$ containing $\sigma$. In the following,  such a simplex $\widetilde \sigma$ will be said to have maximal intrinsic dimension.  Let us equip each such $\widetilde \sigma$ with the usual Lebesgue measure, and together they give a measure  
	$\mu$ on $S$. 
	
	Denote $G_y(x)=G(x,y)$ for $x\in M\backslash S$ and $y\in S$.	Set
	$$u(x)=\int_S G_z(x)d\mu(z).$$
	By construction, $u$ is positive and smooth on $M\backslash S$ with $\mathcal Lu=0$. Now we estimate the asymptotic behavior of $u(x)$ when $x\to S$. Assume that $d(x,S)=r$, and $y\in S$ is a point at which $d(x,y) = d(x, S)$. Let $B_{S}(y;a) \subset S$ be the metric ball in $S$ centered at $y$ with radius $a>0$. If $a$ is sufficiently small, we have
	$$d(x,z)\sim\sqrt{r^2+d(z, y)^2}$$
	for any $z\in B_{S}(y; a)$. Therefore, 
	\begin{align*}
		u(x)\sim&\int_{B_{S}(y; a)}G(x,z)d\mu(z)\sim \int_{B_{S}(y; a)} d(x,z)^{2-n}d\mu(z)\\
		= & \sum \int_{B_{S}(y; a)\cap \widetilde \sigma_j}  d(x,z)^{2-n}d\mu(z) \sim \sum  \int_0^a\frac{1}{(r^2+t^2)^\frac{n-2}{2}}t^{\dim(\widetilde \sigma_j)-1}dt.
	\end{align*}
	where $\widetilde \sigma_j$ runs through all simplices of $S$ that have maximal intrinsic dimension. 	By assumption,  $S$ has codimension $k$, thus $\dim(\widetilde \sigma_j)\leq n-k$. A straightforward computation shows that 
	\begin{align*}
		\int_0^a\frac{1}{(r^2+t^2)^\frac{n-2}{2}}t^{n-\ell-1}dt=r^{2-\ell}\int_{0}^{\frac a r}\frac{s}{(1+s^2)^{\frac{n-2}{2}}}ds
		\sim\begin{cases}
			1& \textup{ if } \ell =0,1\\
			|\log(r)|& \textup{ if } \ell=2, \\
			r^{2-\ell} &  \textup{ if } \ell \geq 3.
		\end{cases}
	\end{align*}
	Now for each $\widetilde \sigma_j$ as above, if $\dim(\widetilde \sigma_j) = n -\ell$, then $\ell\geq k \geq 3$.  This finishes the proof.
\end{proof}

\begin{proposition}\label{prop:conformal}
	Let $M^n$ ($n\geq 3$) be a closed spin manifold and $S$ a finite simplicial complex embedded in $M$ of codimension $k$. Let $g$ be an $L^\infty$-metric on $M$ that is smooth outside $S$. Let $f\colon (M,g)\to (\sph^n,g_{\sph^n})$ be a Lipschitz map with non-zero degree that is  smooth away from $S$. Suppose that
	\begin{equation}
		\Sc_{g}\geq \sigma_{f}=|\Bigwedge^2df|\cdot n(n-1)
	\end{equation} on $M\backslash S$.
	If $f$ is not an isometry up to scaling and $k\geq\frac{n}{2}+1$, then there exist a complete smooth Riemannian metric $g_\times$ on  $M\backslash S$ and a smooth map $f_\times\colon (M\backslash S,g_\times)\to (\sph^n,g_{\sph^n})$ with non-zero degree such that $$\Sc_{g_\times}>\sigma_{f_\times},$$ 
	and  $f_\times(M_\times\backslash K)$ is contained in $f(S)$ for some compact set $K$ of $M_\times$.
\end{proposition}
\begin{proof}
	By Lemma \ref{lemma:constant}, we may assume  without loss of generality  that $\Sc_{g}\geq \sigma_{f}+\delta$ on $M\backslash S$, and $f(U)\subset f(S)$ for some open neighborhood $U$ of $S$.
	
	Set
	$$g_{\times}=(1+u)^{\frac{4}{n-2}}g,$$
	where $u$ is the function constructed in Lemma \ref{lemma:conformaFactor}. In particular, on $M\backslash S$, we have
	$$0=\mathcal Lu=-\frac{4(n-1)}{n-2}\Delta_gu+\phi(\Sc_g-\sigma_f )u.$$
	Denote by $M_\times$ the open manifold $(M\backslash S,g_\times)$ and $f_\times \coloneqq f \colon  (M_\times, g_\times)\to (\sph^n,g_{\sph^n})$.
	The scalar curvature of $g_\times$ is given by
	$$\Sc_{g_\times}=(1+u)^{-\frac{n+2}{n-2}}\Big(-\frac{4(n-1)}{n-2}\Delta_g(1+u)+\Sc_g(1+u)\Big).$$
	It follows that
	\begin{align*}
		\Sc_{g_\times}=&(1+u)^{-\frac{n+2}{n-2}}\Big(-\phi(\Sc_{g}-\sigma_f)u+\Sc_g+\Sc_gu\Big)\\
		\geq &\frac{\Sc_g+\sigma_fu}{(1+u)^{\frac{n+2}{n-2}}}
		=\sigma_{f_\times}+\frac{\Sc_g-\sigma_f}{(1+ u)^{\frac{n+2}{n-2}}}\\
		\geq& \sigma_{f_\times}+\frac{\delta}{(1+u)^{\frac{n+2}{n-2}}}.
	\end{align*}
	Here we have used the fact that $\sigma_{f_\times}=\sigma_f(1+u)^{-\frac{4}{n-2}}$.
	
	Now it remains to show  $(M_\times, g_\times)$ is complete. Let $x\in M\backslash S$ be a point with  $d_g(x,S)=r$ for  some sufficiently small $r>0$, and $y\in S$ a point such that $d_g(x, y) = d_g(x, S)$. Here $d_g$ stands for the distance function with respect to the metric $g$. To show  $g_\times$ is complete, it suffices to show $d_{g_\times}(x,y)=\infty$ for all such $x$ and $y$. Note that along the any path $\gamma$ from $x$ to $y$, we have 
	$$u(z)\geq C d_g(z, y)^{2-k}$$
	for some $C>0$ and all $z\in \gamma$. 
	Therefore,
	$$d_{g_\times}(x,y)\geq C\int_0^r (1+t^{2-k})^{\frac{2}{n-2}}dt\sim\int_0^r \frac{1}{t^{\frac{2(k-2)}{n-2}}}dt.$$
	Note that if $\frac{2(k-2)}{n-2}\geq 1,$ or equivalently, $k\geq \frac n 2+1$,   then   
	\[   \int_0^r \frac{1}{t^{\frac{2(k-2)}{n-2}}}dt = \infty.\]
	This finishes the proof.
\end{proof}
\subsection{Proof of Theorem \ref{thm:sphere}}
In this subsection, we prove Theorem \ref{thm:sphere} by using the conformal metric from Proposition \ref{prop:conformal}.

The notion of subsets of spheres with wrapping property was introduced in \cite{Xie:2021tm} for studying related scalar curvature rigidity questions on spheres.
\begin{definition}
	Let $\Sigma$ be a closed subset of $\sph^n$. We say that $\Sigma$ satisfies the wrapping property if there exists a $1$-Lipschitz map $h\colon \sph^n\to\sph^n$ such that
	\begin{enumerate}
		\item $\deg(h)\ne 1$, and
		\item when restricted to any connected component of $\Sigma$, $h$ is equal to an isometry in $\SO(n+1)$.
	\end{enumerate}
\end{definition}

Roughly speaking, the wrapping property of a subset in $\sph^n$ guarantees the existence of two vector bundles over $\sph^n$ that represent different multiples of the fundamental class in $H^n(\sph^n)$, but agree on $\Sigma$.
Here are some examples of subsets of spheres that have with the wrapping property.
\begin{lemma}\label{lemma:folding}
	If $\Sigma$ is a closed subset of $\sph^n$ contained in an open hemisphere, then $\Sigma$ satisfies  the wrapping property.
\end{lemma}
\begin{proof}
Without loss of generality, assume that $\Sigma\subset \sph^n$ is a closed set that is contained in the hemi-sphere $\{(x_1,x_2,\ldots,x_{n+1})\in \R^{n+1}:x_1>0\}$. Let $h$ be the wrapping map 
	$$h(x_1,x_2,\ldots,x_{n+1})\coloneqq \begin{cases}
		(x_1,x_2,\ldots,x_{n+1})& \textup{ if } x_1\geq 0,\\
		(-x_1,x_2,\ldots,x_{n+1})& \textup{ if } x_1<0.
	\end{cases}$$
	Clearly the map $h$ is a $1$-Lipschitz that has degree zero. Moreover, $h|_\Sigma$ is equal to the identity map. It follows that $\Sigma$ has wrapping property.
\end{proof}

By repeatedly applying wrapping maps such as $h$ above,  one can show that any finite subset of $\sph^n$ has the wrapping property. The proof is elementary, but a bit tedious. We defer it to the Appendix  \ref{app:wrapping}, cf. \cite{Xie:2021tm}.

Now we are ready to prove  Theorem \ref{thm:sphere}.
\begin{proof}[Proof of Theorem \ref{thm:sphere}]
	We first consider the case when the dimension $n$ is even.
	Assume on the contrary that $f\colon M\to\sph^n$ is not an isometry up to scaling. By Proposition \ref{prop:conformal}, there is a complete Riemannian manifold $(M_\times,g_\times)$ together with nonzero degree map $f_\times\colon M_\times\to\sph^n$ such that $\Sc_{g_\times}>\sigma_{f_\times} = 2|\Bigwedge^2df_\times|_{\tr}$, and  $f_\times(M_\times\backslash K)$ is contained in $f(S)$ for a compact set $K$ of $M_\times$.  
	
	By assumption, $\Sigma = f(S)$ satisfies the wrapping property.  Let $h\colon \sph^n\to\sph^n$ be the $1$-Lipschitz map for the set $\Sigma$ from Definition \ref{def:wrapping}. Consider the map $h\circ f_\times\colon M_\times\to\sph^n$. Since $h$ is $1$-Lipschitz, we still have $\Sc_{g_\times}>\sigma_{h\circ f_\times}$. 
	
	
	We define
	$$E_1=S(TM_\times)\otimes f_\times^*S(T\sph^n),$$
	and
	$$E_2=S(TM_\times)\otimes (h\circ f_\times)^*S(T\sph^n).$$
	Let $E=E_1\oplus  E_2$ on $M_\times$.	Denote by $\overbar c$ and $c$ the Clifford action of $TM_\times$ and $T\sph^n$, respectively, and $\nabla$ the spinorial connection on $E$.
	
	The grading operator ${\mathscr E}$ on $E$ is given by 
	\[ \mathscr E\coloneqq \begin{pmatrix}
		\mathscr E_1 & 0 \\ 
		0 & -\mathscr E_2
	\end{pmatrix}, \]
	where $\mathscr E_i$ is the natural grading on $E_i$. We remark that $- \mathscr E_2$ is the negative of grading operator on $E_2$. In other words, we have reversed the $\pm$ parts of $E_2$.  Consider the Dirac operator 
	\[  D = \begin{pmatrix}
		D_1 & 0 \\
		0 & -D_2
	\end{pmatrix},\]
	where $D_i$ be the associated Dirac operator on $E_i$.  
	
	Let $\{Y_\ell\}_{1\leq \ell \leq m}$ be the connected components of  $M_\times\backslash K$. By Definition \ref{def:wrapping}, since $h$ is an isometry in $\SO(n+1)$ on each connected component of $f(S)$, it induces a bundle isometry 
	$V_\ell\colon E_1\to E_2$ over $Y_\ell$ for each $1\leq \ell\leq m$. Consider the operator 
	\[ U_\ell = \begin{pmatrix}
		0 & V_\ell^\ast \\ 
		V_\ell & 0 
	\end{pmatrix}\]
	on $E$ over $Y_\ell$. By construction, we have  
	$$U_\ell^2=1, ~\mathscr E U_\ell = - U_\ell \mathscr E, \text{ and } U_\ell D+  D U_\ell=0.$$
	
	For each  $1\leq \ell\leq m$, we choose a smooth truncation function $\chi_\ell$ on $M_\times$ such that $\chi_\ell$ is supported in $Y_\ell$, and $\chi_\ell|_{Z_\ell}\equiv 1$ for some $Z_\ell\subset Y_\ell$ such that $Y_\ell-Z_\ell$ is compact.   For any $\varepsilon>0$, we define
	\begin{equation}\label{eq:diracpotential}
B\coloneqq  D+\varepsilon\sum_{\ell=1}^m\chi_\ell  U_\ell.
	\end{equation}
	Since 
	$$B^2= D^2+\varepsilon^2\sum_{\ell=1}^m\chi_\ell^2+\varepsilon\sum_{\ell=1}^m[D,\chi_\ell] U_\ell$$
	 is invertible outside a compact subset of $M_\times$, the operator $B$ is Fredholm. Furthermore, it follows from the relative index theorem that
	$$\ind(B)=2(\deg f-\deg (h\circ f))=2(1-\deg h)\deg f\ne 0.$$
	See Lemma \ref{lemma:fredholm} below.
	
	On the other hand, 
	the Bochner--Lichnerowicz--Weitzenb\"{o}ck formula \cite[II, Theorem 8.17]{spingeometry} yields that
	$$ D_1^2=\nabla^*\nabla+\frac{\Sc_{g_\times}}{4}-\frac 1 2\sum_{i\ne j}\overbar c(e_i)\overbar c(e_j)\otimes c((f_{\times})_\ast e_i\wedge  (f_{\times})_\ast e_j)$$
	and 
	$$ D_2^2=\nabla^*\nabla+\frac{\Sc_{g_\times}}{4}-\frac 1 2\sum_{i\ne j}\overbar c(e_i)\overbar c(e_j)\otimes c((h\circ f_{\times})_\ast e_i\wedge  (h\circ f_{\times})_\ast e_j)$$
	for a local orthonormal basis $\{e_i\}$ of $TM_\times$, where the Clifford action of $2$-forms is given by 
	$$c(u\wedge v)=\frac{c(u)c(v)-c(v)c(u)}{2},~u,v\in T\sph^n.$$
	It follows that 
	$$D^2\geq \frac{\Sc_{g_\times}-\sigma_{ f_\times}}{4}$$
where we have used the fact that 
	$\sigma_{h\circ f_\times}\leq \sigma_{f_\times}$.
	
	Note that $[D,\chi_\ell]=\overbar c(\nabla\chi_\ell )$. We obtain that
	$$B^2\geq \frac{\Sc_{g_\times}-\sigma_{ f_\times}}{4}+\varepsilon^2\sum_{\ell=1}^m\chi_\ell^2-\varepsilon\sum_{\ell=1}^m|\nabla\chi_\ell|.$$
	By construction, each $\nabla\chi_\ell$ is   supported on the compact set $Y_\ell - Z_\ell$. Since by assumption $\Sc_{g_\times}-\sigma_{f_\times}>0$ on $M_\times$,  it follows that there exists  $\delta>0$ such that $\Sc_{g_\times}-\sigma_{f_\times} > \delta$ on the compact closure of  $M_\times-\bigcup_{\ell=1}^m Z_\ell$. Hence there exists $\varepsilon>0$ such that
	\begin{equation}\label{eq:lowerbound}
		\frac{\Sc_{g_\times}-\sigma_{f_\times}}{4}-\varepsilon\sum_{\ell=1}^m|\nabla\chi_\ell|\geq \frac{\delta}{8}
	\end{equation}
	on the closure of $M_\times-\bigcup_{\ell =1}^m Z_\ell$. Now on each $Z_\ell$, we have $B^2\geq\varepsilon^2$,  since $\nabla\chi_\ell=0$ and $\Sc_{g_\times}-\sigma_{f_\times}\geq 0$.
	It follows that $B^2\geq\min\{\delta/8,\varepsilon^2\}>0$, in particular, $B$ is invertible hence has $\ind(B) = 0$. This contradicts the fact that $\ind(B)\neq 0$, hence finishes the proof for the even dimensional case. 
	
		The odd dimensional case follows by  the a similar argument except the relative index theory computation uses spectral flow instead, cf. \cite{MR4739559}. 
		 More precisely, for the odd dimensional case, we consider $E\coloneqq E_1 \oplus E_2$ with  
		 $$ E_1= S(TM_\times)\otimes  f_{\times}^*S^+(T\R^{n+1}) \textup{ and } E_2= S(TM_\times)\otimes (h\circ f_{\times})^*S^+(T\R^{n+1}) ,$$
		 where $S^+(T\R^{n+1})$ is the positive part of the spinor bundle of the trivial bundle $T\R^{n+1}$ over $\sph^n$. Let $\partial_r$ be the unit normal vector of $\sph^n\subset\R^{n+1}$. Consider the family of connections $\{\nabla^s\}_{s\in [0, 1]}$ on $E$ over $ M_\times$ given by
		 $$\nabla_X+s\otimes c(\partial_r)c( (f_{\times})_\ast X) \textup{ on } E_1$$
		 and $$\nabla_X+s\otimes c(\partial_r)c((h\circ f_{\times})_\ast X) \textup{ on } E_2$$
		 for all $X\in TM_\times$, where $\nabla$ is the natural spinorial connection on $E$. Clearly, the connection $\nabla^s$ preserves the Riemannian metric on $E$, namely, 
		 $$X\langle\sigma_1,\sigma_2\rangle=\langle\nabla^s_X\sigma_1,\sigma_2\rangle+\langle\sigma_1,\nabla^s_X\sigma_2\rangle,~\forall \sigma_1,\sigma_2\in C^\infty( E)\textup{ and } X\in TM_\times, $$
		 and preserves the Clifford action, namely, 
		 $$\nabla^s_X\big(c(Y)\sigma\big)=c(\nabla_X^{\text{LC}} Y)\sigma+c(Y)\nabla^s_X\sigma,~\forall \sigma\in C^\infty(E) \textup{ and } X,Y\in TM_\times,$$
		 where  $\nabla^{\text{LC}}$ is the Levi-Civita connection on $ M_\times$. By applying the fact that 
		 \[ \nabla_X(c(\partial_r)) = c((f_{\times})_*X) \textup{ on } E_1 \textup{ and } \nabla_X(c(\partial_r)) = c((h\circ f_{\times})_*X) \textup{ on } E_2,  \] a direct computation shows that  the curvature operator of $\nabla^s$ is given by
		 $$R^s(e_i,e_j)=R(e_i,e_j)-2s(1-s)\otimes c((f_{\times})_*e_i\wedge (f_{\times})_* e_j) \textup{ on } E_1$$
		 and 
		 $$R^s(e_i,e_j)=R(e_i,e_j)-2s(1-s)\otimes c((h\circ f_{\times})_*  e_i\wedge (h\circ f_{\times})_* e_j) \textup{ on } E_2, $$
		 where $R$ is the curvature of $\nabla$.
		 
		 Consider the Dirac operator 
		 $$ D_1^s=\sum_{i=1}^{n}\overbar c(e_i)\nabla_{e_i}^s$$
		 acting on $E_1$ over $M_\times$, and similarly $D_2^s$ acting on $E_2$. 
		 
		 Let 
		 \[ D^s = \begin{pmatrix}
		 	D_1^s & 0 \\ 0 & -D_2^s
		 \end{pmatrix} \]
		  Similar to the even dimensional case above, since $h$ is an isometry in $\SO(n+1)$ on each connected component of $f(S)$, $h$ induces a bundle isometry of $S^+(T\R^{n+1})$ to $S^+(T\R^{n+1})$, which in turn gives a bundle isometry 
		 $V_\ell\colon E_1\to E_2$ over $Y_\ell$ for each $1\leq \ell\leq m$. Consider the operator 
		 \[ U_\ell = \begin{pmatrix}
		 	0 & V_\ell^\ast \\ 
		 	V_\ell & 0 
		 \end{pmatrix}\]
		 on $E$ over $Y_\ell$.
		 Set
		 	$$B^s\coloneqq  D^s+\varepsilon\sum_{l=1}^m\chi_\ell  U_\ell.$$
		 By construction, we have $U_\ell D^s+  D^s U_\ell=0$.

		 Now it follows from the Bochner--Lichnerowicz--Weitzenb\"{o}ck formula  that 
		 $$(D_1^s)^2=\nabla^{s*}\nabla^s+\frac{\Sc_{g_\times}}{4}-s(1-s)\sum_{i\ne j}\overbar c(e_i)\overbar c(e_j)\otimes  c((f_{\times})_*e_i\wedge (f_{\times})_* e_j) $$
		 and 
		 $$(D_2^s)^2=\nabla^{s*}\nabla^s+\frac{\Sc_{g_\times}}{4}-s(1-s)\sum_{i\ne j}\overbar c(e_i)\overbar c(e_j)\otimes c((h\circ f_{\times})_*  e_i\wedge (h\circ f_{\times})_* e_j).$$ 
		 By a similar argument as in the even dimensional case, we see that  $B^s$ is invertible for all $s\in[0,1]$. However, the spectral flow of the family $\{B^s\}_{s\in [0, 1]}$ is nonzero. which follows for example by the argument from \cite{MR4739559} combined together the argument in  Lemma \ref{lemma:fredholm} below. This leads to a contradiction and completes the proof of the odd dimensional case.	
\end{proof}

Let us now prove Corollary \ref{coro:point&tree}.

\begin{proof}[Proof of Corollary \ref{coro:point&tree}] We prove the corollary case by case. 
	
	\textbf{Case (1).} If $\Sigma = f(S)$ is contained in an open hemi-sphere, then $\Sigma$ has the wrapping property by Lemma \ref{lemma:folding}, thus the conclusion immediately follows from Theorem \ref{thm:sphere}.
	Suppose that $\Sigma$ is contained in a closed hemi-sphere and $f$ is not an isometry up to scaling. Let $\delta_2>0$ and $f_2\colon (M,g_2)\to \sph^n$ be as in Lemma \ref{lemma:constant}. We parametrize $\sph^n$ in polar coordinates $(r,\Theta)$ with $0\leq r\leq \pi$ and $\Theta\in\sph^{n-1}$. Let $\rho\colon [0,\pi]\to [0, \pi]$ be a smooth function such that
	\begin{itemize}
		\item $\rho(r)\equiv r$ near $0$ and $\pi$, $0<\rho'<1+\delta_2'$, and 
		\item $\rho(\pi/2)<\pi/2$.
	\end{itemize}
	 Consider the following map
	$$h\colon \sph^n\to\sph^n,~(r,\Theta)\mapsto (\rho(r),\Theta).$$
	Then $f_2'\coloneqq h\circ f_2\colon (M,g_2)\to \sph^n$ still satisfies the inequalities as in Lemma \ref{lemma:constant}, and $\Sigma'\coloneqq f_2'(S)$ is contained in an open hemi-sphere.
	 In particular, Proposition \ref{prop:conformal} and the same proof of Theorem \ref{thm:sphere} still apply to give the conclusion. 
	
	\textbf{Case (2).} By Proposition \ref{prop:wrapfin}, each finite subset of $\sph^n$ satisfies the wrapping property. Hence this case follows immediately from Theorem \ref{thm:sphere}.
	
	\textbf{Case (3).} Let us consider the case where $\Sigma$ is a finite union of trees.  Suppose that $f$ is not an isometry up to scaling. By Proposition \ref{prop:conformal}, there exists a complete manifold $(M_\times,g_\times)$ and $f_\times\colon (M_\times,g_\times)\to(\sph^n,g_{\sph^n})$, such that $\Sc_{g_\times}>\sigma_{f_\times}$. Furthermore, there is a compact set $K\subset M_\times$ such that $f_\times(M_\times\backslash K)$ lies in  $\Sigma$.
	
	Since $\Sigma$ is a disjoint union of trees, there exists a smooth map 
	$$H\colon [0,+\infty)\times \Sigma\to \Sigma$$
	such that $H(t,\cdot)$ is the identity map for $t\leq 1$, and $H(t,\cdot)$ is a locally constant map for $t\geq 2$. Let $r$ be the smooth non-negative function on $M_\times$ that is almost the distance from $K$. We define a new map $\widetilde f_\times\colon M_\times\to\sph^n$ by
	$$\widetilde f_\times(x)=\begin{cases}
		f_\times(x)& \textup{ if } x\in K,\\
  H(r(x),f_\times(x))& \textup{ if } x\in M_\times \backslash K. 
	\end{cases}$$ 
	By construction, $\widetilde f_\times$ is a locally constant map outside a compact set $\widetilde K$, where $\widetilde K$ is  the $2$-neighborhood of $K$. Since $\Sigma$ is one dimensional, we have  $\sigma_{\widetilde f_\times}= \sigma_{f_\times}=2|\Bigwedge^2d(\widetilde f_\times)|_{\tr}= 0$ on $M_\times\backslash K$. In particular, we still have $\Sc_{g_\times}>\sigma_{\widetilde f_\times}$. So we have reduced this case to Case (2), the case where $\Sigma$ is a finite set. This completes the proof
\end{proof}

Let us finish this section with the following lemma of index computation.

\begin{lemma}\label{lemma:fredholm}
	For any $\varepsilon>0$, the operator 
	$$B\coloneqq  D+\varepsilon\sum_{\ell=1}^m\chi_\ell  U_\ell$$
	from line \eqref{eq:diracpotential}  is a Fredholm operator with \[  \ind(B)=2(\deg(f)-\deg(hf)). \]
\end{lemma}
\begin{proof}
	Note that
	$$B^2= D^2+\varepsilon^2\sum_{\ell =1}^m\chi_\ell^2+\varepsilon\sum_{\ell=1}^m[D,\chi_\ell] U_\ell.$$
	By construction, $\chi_\ell\equiv 1$ and $[D,\chi_\ell]=\overbar c(\nabla\chi_\ell)=0$ on $Z_\ell$. Therefore, $B^2\geq \varepsilon^2$ outside the compact set $M_\times\backslash\cup_{\ell=1}^m Z_\ell$. Thus $B$ is Fredholm.
	
	By construction, $B$ is an odd operator with respect to the grading $\mathscr E$. Let us compute the index of $B$ by the relative index theorem of Gromov and Lawson \cite{GromovLawson}. Recall that the sets $Z_\ell$ are connected components of $\mathcal N\backslash S$, where $\mathcal N$ is some small neighborhood of $S$ in $M$. In particular,  $\mathcal N$ itself is a smooth manifold with smooth boundary. 
	
	Let us define 
	$$\widetilde Z\coloneqq (\cup_{\ell=1}^m Z_\ell)\bigcup_{\partial \mathcal N } \mathcal N  \textup{ and } \widetilde M\coloneqq\big( M_\times\backslash\cup_{\ell=1}^m Z_\ell \big)\bigcup_{\partial \mathcal N} \mathcal N.$$
	Extend the metric $g_\times$ on $\cup_{\ell=1}^m Z_\ell$ (resp. $M_\times\backslash\cup_{\ell=1}^m Z_\ell$) to a Riemmanian metric $g_1$ (resp. $g_2$) on $\widetilde Z$ so that  $g_1$ (resp. $g_2$) is  smooth on $\mathcal N$. Let  $\widetilde f_{\mathcal N}\colon \mathcal N\to \sph^n$ be a smooth approximation of the original map $f$ that agrees with $f$ near $\partial \mathcal N$. Then $\widetilde f$ together with $f_\times$ defines a map $f_1\colon \widetilde Z \to \sph^n$ (resp. $f_2\colon \widetilde M \to \sph^n$).  
	
	Consider the bundle 
	\[   E_{\widetilde Z} = S(T\widetilde Z)\otimes f_1^\ast S(T\sph^n) \oplus S(T\widetilde Z)\otimes (h\circ f_1)^\ast S(T\sph^n)  \]
	over $\widetilde Z$ and 
	\[   E_{\widetilde Z} = S(T\widetilde M)\otimes f_2^\ast S(T\sph^n) \oplus S(T\widetilde M)\otimes (h\circ f_2)^\ast S(T\sph^n)  \]
	over $\widetilde M$.  By extending the functions $\chi_\ell$ smoothly over $\widetilde Z$ and  $\widetilde M$ respectively, we define 
	\[ B_{\widetilde Z} = D_{\widetilde Z}+\varepsilon\sum_{\ell=1}^m\chi_\ell  U_\ell \]
	and  
		\[ B_{\widetilde M} = D_{\widetilde M}+\varepsilon\sum_{\ell=1}^m\chi_\ell  U_\ell. \]
It follows from the relative index theorem \cite{GromovLawson} that   
	$$\ind(B)=\ind(B_{\widetilde M})-\ind( B_{\widetilde Z}).$$
By a linear homotopy, we see that 
	$$\ind(B_{\widetilde Z})=\ind( D_{\widetilde  Z}+\sum_{\ell=1}^m U_\ell),$$
Note that $D_{\widetilde  Z}+\sum_{\ell=1}^m U_\ell$ is invertible, hence has zero index. 
On the other hand, since $\widetilde M$ is a closed manifold, we have 
	$$\ind(B_{\widetilde M})=\ind( D_{\widetilde M})=\deg(f)\chi(\sph^n)-\deg(hf)\chi(\sph^n).$$
	This finishes the proof.
\end{proof}

\section{Scalar curvature rigidity for punctured torus}\label{sec:tori}

In this section, we shall prove the following analogue of Theorem \ref{thm:sphere} for tori.
\begin{theorem}[Theorem \ref{thm:torus-intro}]
Let $M^n$ ($n\geq 3$) be a closed spin manifold and $S$ a finite simplicial complex embedded in $M$ of codimension $k$. Let $g$ be an $L^\infty$-metric on $M$ that is smooth outside $S$. Let $f\colon (M,g)\to (\T^n,g_{\T^n})$ be a Lipschitz map with non-zero degree. Suppose that $\Sc_{g}\geq 0$ on $M\backslash S$. If $k = \codim S  \geq\frac n 2+1$ and $$(f\circ i)_*\colon \pi_1(S)\to \pi_1(\T^n)$$ is the zero map, where $i\colon S\to M$ is the inclusion map,
then $g$ is Ricci flat on $M\backslash S$.
\end{theorem}
\begin{proof}
	We only prove the case when $n$ is even. The odd dimensional case follows by a similar argument by using spectral flow (cf. the proof of the odd dimensional case of Theorem \ref{thm:sphere}).
	
	Assume to the contrary that $g$ is not Ricci flat. By \cite[Theorem B]{Kazdan82} and Lemma \ref{lemma:bound}, there exists an $L^\infty$-metric $g_1$ on $M$ such that $g_1$ is smooth outside $S$ and  $\Sc_{g_1}\geq\delta>0$ for some $\delta>0$. Hence we may assume without loss of generality that  $\Sc_{g}\geq\delta>0$ for some $\delta>0$. In this case, we denote by $L$ the Lipschitz constant of $f\colon (M,g)\to (\T^n,g_{\T^n})$.
	
	Denote
	$$\mathcal L=-\frac{4(n-1)}{n-2}\Delta_g+\phi(\Sc_g),$$
	where $\phi\colon \R_{\geq 0}\to\R_{\geq 0}$ is a smooth non-decreasing bounded function such that
	$$\phi(s)=\begin{cases}
		s,&s\leq 10\\
		20,&s\geq 30
	\end{cases}$$
	and $\phi(x)\leq x$. Similar to Lemma \ref{lemma:conformaFactor},  there exist a positive smooth function $u$ on $M\backslash S$ and a continuous positive function $\beta\colon S\to [k, n]$ such that $\mathcal L u=0$ on $M\backslash S$
	and 
	\begin{equation*}
		\frac 1 C d(x,S)^{2-\beta(x)}\leq u(x)\leq Cd(x,S)^{2-\beta(x)}
	\end{equation*}
	for some $C>0$.
	
 We consider the conformal metric
	$$g_\times=(1+u)^{\frac{4}{n-2}}g.$$
	By Proposition \ref{prop:conformal}, $(M\backslash S,g_\times)$ is a complete metric,  provided that $k\geq n/2 +1$.
	The scalar curvature of $g_\times$ is given by
	\begin{align*}
		\Sc_{g_\times}=&(1+u)^{-\frac{n+2}{n-2}}\Big(-\frac{4(n-1)}{n-2}\Delta_g(1+ u)+\Sc_g(1+ u)\Big)\\
		=&(1+u)^{-\frac{n+2}{n-2}}\Big(-\phi(\Sc_{g})u+\Sc_g+\Sc_gu\Big)\\
			\geq &\frac{\Sc_g}{(1+u)^{\frac{n+2}{n-2}}}\geq  \frac{\delta}{(1+ u)^{\frac{n+2}{n-2}}}.
	\end{align*}
	
	Recall that the torus $\T^n$ is enlargeable \cite{GromovLawson80}, namely,  for any $\alpha>0$, there exists a finite covering space $\widetilde \T^n$ and a smooth map $\psi\colon (\widetilde \T^n,\widetilde g_{\T^n})\to(\sph^n,g_{\sph^n})$ with non-zero degree such that $\psi$ has Lipschitz constant $\leq \alpha$, where $\widetilde g_{\T^n}$ is the metric on $\widetilde\T^n$ that is lifted from $g_{\T^n}$.  Let  $\widetilde M$ is the covering space of $M$ induced by $\widetilde\T^n\to\T^n$ via the map $f$, and $\widetilde f\colon \widetilde M\to \widetilde\T^n$ be the lifted map. Let $\widetilde S$ be the preimage of $S$ in $\widetilde M$ under the covering map.  Since $ f_*\colon \pi_1(S)\to \pi_1(\T^n)$ is the zero map, $\widetilde S$ is a disjoint union of finitely many copies of $S$. It follows that each copy $S$ in $\widetilde S$ is mapped by $\psi \circ \widetilde f$ to an arbitrarily small set in $\sph^n$, provided the covering space $\widetilde \T^n$ is sufficiently large. 
 
 Let  $\widetilde M_\times$ be the covering space of $M_\times$ induced by $\widetilde \T^n\to \T^n$ via the map $f_\times$ 
 and $\widetilde f_\times\colon \widetilde M_\times\to\widetilde\T^n$ the lifted map. The metric $g_\times$ lifts to a metric  $\widetilde g_\times$   on $\widetilde M_\times$. There exists a compact set $K$ of $M_\times$, whose preimage in $\widetilde M_\times$ will be denoted by $\widetilde K$, such that  each connected component $Y_\ell$ of $\widetilde M_\times-\widetilde K$ is mapped by the map $\psi\circ \widetilde f_\times $ into a small subset in $\sph^n$ that has diameter $\leq c\alpha$, where $c>0$ is a constant independent of $\alpha$. 
	
	We shall modify the map $\psi\circ \widetilde f_\times$ to reduce the proof to a similar situation as Case (2) of Corollary \ref{coro:point&tree}. Let $r$ be a smooth function on $\widetilde M_\times$ that is a smooth approximation of the distance from $\widetilde K$. Assume without loss of generality that $r$ is a $2$-Lipschitz function. As long as  $\alpha$ is sufficiently small,  there exists a $1$-Lipschitz homotopy
	$$H_\ell\colon [0,1]\times \psi \widetilde f_\times(Y_\ell)\to\sph^n,$$
	such that $H_\ell(0,\cdot)$ is the identity map on $\widetilde M_\times-\widetilde K$, and $H_\ell(1,\cdot)$ is a constant map. 
	Let $\rho_\alpha\colon[0,\infty)\to[0,1]$ be a non-decreasing function such that $\rho(t)=0$ for $t\leq 1$, $\rho(t)=1$ for $t\geq 2\alpha^{-1}$, and $0\leq \rho'\leq\alpha$.
	Now we define
	$$F\colon\widetilde M_\times\to\sph^n \textup{ by } F(x)=\begin{cases}
	\psi \circ\widetilde f_\times(x),&\textup{ if } x\in\widetilde K,\\
	H_\ell (\rho_\alpha(r(x)),h\circ \widetilde f_\times(x)),&\textup{ if } x\in Y_l.	
	\end{cases}$$
	
	Let $K_\alpha$ be the $3\alpha$-neighborhood of $K$ in $M_\times$ and $\widetilde K_\alpha$ its lift in $\widetilde M_\times$. By construction, $F$ has non-zero degree and  is locally constant on $\widetilde M_\times-\widetilde K_\alpha$.
	\begin{claim*}
		If $\alpha$ is sufficiently small, then $\Sc_{\widetilde g_\times}>\sigma_F$, where $\sigma_F=2|\Bigwedge^2dF|_\tr$.
	\end{claim*}
	If the claim holds, then we obtain a contradiction by the same proof of Case (2) of Corollary \ref{coro:point&tree}. This will then finish the proof of the theorem. Therefore, it remains to prove the claim. 

  Note that we have 
	$$\Sc_{\widetilde g_\times}\geq\frac{\delta}{(1+\widetilde u)^{\frac{n+1}{n-2}}},$$
	where $\widetilde u$ is the lift of $u$ from $M_\times$ to  $\widetilde M_\times$. On $\widetilde M_\times-\widetilde K_\alpha$, we have $\sigma_F=0$ hence $\Sc_{\widetilde g_\times}>\sigma_F$. Now on $\widetilde K_\alpha$, we note that with respect to the original smooth metric $\widetilde g$ on $\widetilde K$, the map
	$$F\colon (\widetilde K,\widetilde g)\to (\sph^n,g_{\sph^n})$$
	has Lipschitz constant  $\leq (L+2)\alpha$, where $L$ is the Lipschitz constant of $f\colon (M, g) \to (\T^n, g_{\T^n})$. Since 
	$$\widetilde g_\times=(1+\widetilde u)^{\frac{4}{n-2}}\widetilde g\geq \widetilde g,$$
	we have
	$$\sigma_F\leq \frac{n(n-1)}{2}(L+2)^2\alpha^2.$$
	It follows that 
	$$\Sc_{\widetilde g_\times}\geq\frac{\delta}{(1+\widetilde u)^{\frac{n+1}{n-2}}}\geq\frac{\delta}{(1+\sup_{x\in K_\alpha}\{u(x)\})^{\frac{n+1}{n-2}}}$$
 on $\widetilde K_\alpha$.
	As $K_\alpha$ is compact, we see that $\Sc_g>\sigma_F$ as long as 
	$$\alpha<\frac{\sqrt\delta}{(L+2)\sqrt{\frac{n(n-1)}{2}}(1+\sup_{x\in K_\alpha}\{u(x)\})^{\frac{n+1}{2(n-2)}}}.$$
	This finishes the proof.
\end{proof}

We remark that the above theorem (and the same proof) also hold if  we replace the torus $\mathbb T^n$  by any closed enlargeable spin manifold (cf.  \cite[Definition 6.4]{GromovLawson}).

\section{Positive mass theorem for $L^\infty$ metrics}\label{sec:pmt}
In this section, we prove a positive mass theorem for complete $L^\infty$ metrics on asymptotically flat spin manifolds with arbitrary ends. 

\begin{definition}\label{def:af}
Let $M$ be a smooth manifold with no boundary equipped with an $L^\infty$ Riemannian metric $g$. We say $(M, g)$ is a complete asymptotically flat manifold with arbitrary end if 
	\begin{enumerate}
		\item there exists a compact set $K\subset M$ such that $g$ is smooth on $M\backslash K$, 
		\item there is a distinguished end $\mathcal E$ such that $\mathcal E$ is diffeomorphic to $\R^n\backslash B_R(0)$ for some $R>0$, and under this diffeomorphism the smooth metric $g$ satisfies
		$$g=g_{\R^n}+O(\rho^{2-n}),~\Sc_g=O(\rho^{-q}),~|(\nabla^{\R^n})^ig|=O(\rho^{2-n-i}),$$
		for $i=1,2$, where $\rho$ is the Euclidean distance function.
	\end{enumerate}
\end{definition}

When $(M,g)$ is a complete asymptotically flat manifold with arbitrary ends. The ADM mass (of the distinguished end $\mathcal E$) is defined as follows.
\begin{definition}
 The mass $m_g$ (of the distinguished end $\mathcal E$) of $(M, g)$ is defined by
	$$m_g=\lim_{R\to\infty}\frac{1}{\omega_n}\int_{S_R}\sum_{i,j}(\partial_i g_{ji}-\partial_j g_{ii})\star dx_j,$$
	where $\{x_i\}$ is the coordinate of $\R^n$, $\star$ is the Hodge star operator on Euclidean space, and $S_R$ the sphere of radius $R$ in $\R^n$ with $\omega_n$ its volume.
	\end{definition}

The main theorem of this section is as follows.
\begin{theorem}
	Let $(M^n,g)$ be a complete asymptotically flat spin manifold  with arbitrary ends, and $S$ a finite simplicial complex embedded in $M$ of codimension $k\geq \frac{n}{2} + 1$. Suppose that $g$ is smooth outside $S$. If $\Sc_g\geq 0$, then $m_g\geq 0$. Furthermore, if $m_g=0$, then $g$ is Ricci flat.
\end{theorem}
\begin{proof}
	Assume first $m_g<0$. By the same argument of \cite[Proposition 3.2]{MR4597224}, for any $\varepsilon>0$, there exists  $L^\infty$-metric $g_1$ on $M$ which is smooth outside $S$, such that 
	\begin{itemize}
		\item $\|g_1-g\|_{L^\infty,g}<\varepsilon$,
		\item $\Sc_{g_1}\geq 0$,
		\item $g_1$ is asymptotically flat on on the distinguished end $\mathcal E$, 
		\item the ADM mass of $g_1$ (on the distinguished end $\mathcal E$) is negative, and
		\item $g_1$ is conformally Euclidean outside a compact set, namely
		$$g_1=u^{\frac{4}{n-2}}g_{\R^n},$$
		where $u$ at infinity has the expansion 
		$$u=1+\frac{m_{g_1}}{2\rho^{n-2}}+O(\rho^{1-n}).$$
	\end{itemize}
	In fact, by the standard elliptic theory for Laplacian with bounded coefficients \cite[Theorem 8.6 and Theorem 8.22]{MR1814364}, the Laplacian equation with $L^\infty$-coefficients is still solvable and has solutions in $C^\infty_{\mathrm{loc}}\cap C^\alpha$. Hence the same exhaustion argument in \cite[Proposition 2.2]{MR4597224} still holds. See also \cite[Section 4]{MR994021} and \cite{Bartnik}.
	
	By following Lohkamp's argument \cite{MR1678604}, we modify the function $u$ above on the distinguished end $\mathcal E$ and  obtain a new metric $g_2$ on $M$, so that $\Sc_{g_2}\geq 0$, $\Sc_{g_2}(x) >0$ for some $x\in \mathcal E$,  and $g_2=g_{\R^n}$ on   $\mathcal E\backslash \mathcal K$, where $\mathcal K$ is a precompact subset of $\mathcal E$. Now choose a sufficiently large  sphere $\sph^n_r$ in $\mathcal E\backslash \mathcal K$ for some radius $r>0$. Denote the part of $\mathcal E\backslash \mathcal K$ that is outside of  $\sph^n_r$ by $\mathcal O$. Then $\mathcal O$ is isometric to $(\mathbb R^n\backslash B_r(0), g_{\mathbb R^n})$. Now  consider a large flat torus $\mathbb T^n$ such that $B_r(0)$ isometrically embeds into $\mathbb T^n$. We replace the ball $B_r(0)$ by the set $M\backslash \mathcal O$, and obtain a complete  spin Riemannian manifold $(\overbar M,\overbar g)$, where $\overbar g$ is $L^\infty$ and smooth outside $S\subset \overbar M$ and $\Sc_{\overbar g}\geq 0$. Note that $\overbar M$ is a connected sum of a complete spin Riemannian manifold with a torus, hence enlargeable in the sense of \cite[Definition 6.4]{GromovLawson}. In fact, consider the obvious map $f\colon \overbar M \to \mathbb T^n$ that crashes $M\backslash \mathcal O$ to a point. Clearly, $f$ is a Lipschitz map of non-zero degree that maps   $\pi_1(S)$ to zero in $\pi_1(\mathbb T^n)$.   By applying Theorem \ref{thm:torus-intro} (or rather its proof), we arrive at a contradiction. We remark that, although  Theorem \ref{thm:torus-intro} is only stated for the case where $\overbar M$ is a closed manifold, the same proof in fact applies to the current case of complete manifolds, by introducing auxiliary potentials along various  ends as in the proof of Theorem \ref{thm:sphere}.  To summarize, we have shown that $m_g\geq 0$. 
	
	Now suppose $m_g=0$ but $g$ is not Ricci-flat, then  the same argument from  \cite[Lemma 3]{MR788292} would lead us to a contradiction. This completes the proof. 
\end{proof}
\appendix

\section{Wrapping property for finite sets}\label{app:wrapping}

In this appendix, we show that any finite subset of $\sph^n$ satisfies the wrapping property, cf. \cite{Xie:2021tm}.

\begin{proposition}\label{prop:wrapfin}
	If $\Sigma$ is a finite subset of $\sph^n$, then $\Sigma$ satisfies the wrapping property. 
\end{proposition}

\begin{proof}
	Consider the canonical embedding of the unit sphere $\sph^n$ inside the Euclidean space $\mathbb R^{n+1}$. Since $\Sigma$ is finite, there exists an vector $v$ in $\mathbb R^{n+1}$ such  that 
	\begin{enumerate}
		\item $\langle v, x\rangle \neq 0$ for all $x\in \Sigma$, that is, $v$ is not orthogonal to any vector $x\in \Sigma$;
		\item and $v\nparallel (x-y)$ for all $x, y\in \Sigma$ with $x\neq y$, that is, $v$ is not parallel to any vector $(x-y)$ for all $x, y\in \Sigma$ with $x\neq y$, where $(x-y)$ is viewed as a vector in $ \mathbb R^{n+1}$. 
	\end{enumerate} 
	It follows that there exists an equator $\equator_1$ such that $\equator_1 \cap \Sigma = \varnothing$ and no pair of points in $\Sigma$ are symmetric\footnote{Here we say two points $x_1$ and $x_2$ of $\Sigma$ are symmetric along an equator $\equator$ if the reflection map along $\equator$ takes $x_1$ to $x_2$. } along $\equator_1$. 
	
	Let $\Phi_{\equator_1}$ be the wrapping map along $\equator_1$ as defined in Lemma \ref{lemma:folding}. In particular, due to the above properties (a) and (b) of $\equator_1$, the map $\Phi_{\equator_1}$ is injective on $\Sigma$, hence also injective on a small neighborhood of $\Sigma$. By the construction of $\Phi_{\equator_1}$, on a sufficiently small neighborhood of each point $x\in \Sigma$, the map  $\Phi_{\equator_1}$ coincides with either the identity map or the reflection map along $\equator_1$. Let us introduce an orientation-indicator function $\omega_{1}\colon \Sigma \to \{\pm1\}$ by setting 
	\[ \omega_{1}(x) = \begin{cases}
		+1 & \textup{ if $\Phi_{\equator_1}$ is orientation-preserving on a small neighborhood of $x$,} \\
		-1 & \textup{ if $\Phi_{\equator_1}$ is orientation-reserving on a small neighborhood of $x$.}
	\end{cases} \]
	A point $x\in \Sigma$ with $\omega_1(x) = 1$ will be called a $\omega_1$-positive point, and a point $x\in \Sigma$ with $\omega_1(x) = -1$ will be called a $\omega_1$-negative point. The same terminology also applies to points in $\Phi_{\equator_1}(\Sigma)$. That is, a point $y\in \Phi_{\equator_1}(\Sigma)$ will be called a $\omega_1$-positive (reps. $\omega_1$-negative)  point if $y$ is the image of a $\omega_1$-positive (reps. $\omega_1$-negative) point of $\Sigma$.  This completes the initial step of our mathematical induction. 
	
	After applying $\Phi_{\equator_1}$, the image $\Phi_{\equator_1}(\sph^n)$ lies in a hemisphere. If there are no $\omega_1$-negative points in $\Phi_{\equator_1}(\Sigma)$, then the induction process ends and the proof is completed.   Now suppose the set of $\omega_1$-negative points in $\Phi_{\equator_1}(\Sigma)$ is nonempty.  Let $z_1\in \Phi_{\equator_1}(\Sigma)$ be a closest point to $\equator_1$ among all $\omega_1$-negative points in  $\Phi_{\equator_1}(\Sigma)$, that is, $z_1$ is a $\omega_1$-negative point and 
	\[ \dist(z_1, \equator_1) = \inf_{\substack{y\in \Phi_{\equator_1}(\Sigma) \\ y \textup{ is $\omega_1$-negative}} } \dist(y, \equator_1). \]
	There could be more than one such $z_1$. We simply choose one of them. 
	\begin{claim*}
		There exists an equator $\equator_2$ such that no pair of points in  $\Phi_{\equator_1}(\Sigma)$ are symmetric along $\equator_2$  and furthermore  the  hemispheres $\sph_{\equator_2,\pm}$ associated to $\equator_2$ satisfy the following condition: $z_1$ is contained in $\sph_{\equator_2,-}$ and is the only $\omega_1$-negative point in $\Phi_{\equator_1}(\Sigma)$ that is contained in $\sph_{\equator_2,-}$
	\end{claim*}
	We can find such an equator $\equator_2$ as follows. Denote by $\sph_{\equator_1, \pm}$ the hemispheres determined by $\equator_1$. Say, $\Phi_{\equator_1}(\Sigma)$ is contained in the hemisphere $\sph_{\equator_1, +}$. Let $a_0$ be the center of $\sph_{\equator_1, +}$. If $z_1 = a_0$, then  the set $\Phi_{\equator_1}(\Sigma)$ has one and only one $\omega_1$-negative point, which is $z_1$ itself,  since $a_0$ is the unique point in the hemisphere $\sph_{\equator_1, +}$ to achieve $\dist(a_0, \equator_1) = \frac{\pi}{2}$. Then the existence of an equator $\equator_2$ with the required properties is obvious in this case.  
	
	So without loss of generality, we assume $z_1 \neq a_0$. Let $\gamma$ be the unique geodesic starting at $a_0$, passing through $z_1$ and ending at a point of $\equator_1$. Denote by $v_{z_1}$ the unit tangent vector of the curve $\gamma$ at $z_1$, which is also naturally viewed as a vector in $\mathbb R^{n+1}$. Let $\equator_{z_1}$ be the unique equator that is orthogonal to $v_{z_1}$. Note that $\equator_{z_1}$ passes through the point $z_1$,  and  the two equators $\equator_1$ and $\equator_{z_1}$ are not orthogonal since $z_1\neq a_0$. Let $Q$ be an acute (open) quadrant determined by $\equator_1$ and $\equator_{z_1}$. Then we have 
	\[ \dist(q, \equator_1) < \dist(z_1, \equator_1) \] 
	for all points $q\in Q$. Consequently, we see that $Q$ does not contain any 
	$\omega_1$-negative points of $\Phi_{\equator_1}(\Sigma)$. Now the desired equator $\equator_2$ is obtained by  rotating $\equator_{z_1}$ by a small amount along the geodesic $\gamma$. More precisely, choose a point $y_1\in \gamma$ that is sufficiently close to $z_1$ such that $\dist(y_1, a_0) < \dist(z_1, a_0)$. Let $v_{y_1}$ be the tangent vector of $\gamma$ at $y_1$. Then we can choose $\equator_2$ to be the unique equator that is orthogonal to  $v_{y_1}$ for some $y_1$ that is sufficiently close to $z_1$.  This finishes the proof of the claim.    
	
	Now we return to the induction process. If the hemisphere $\sph_{\equator_2,-}$ contains some $\omega_1$-positive points of $\Phi_{\equator_1}(\Sigma)$, then we  shall first apply a ``double-wrapping" procedure to reduce it to the case where $\sph_{\equator_2,-}$ contains no  $\omega_1$-positive points of $\Phi_{\equator_1}(\Sigma)$, that is, to the case where $\sph_{\equator_2,-}\cap \Phi_{\equator_1}(\Sigma) = \{z_1\}$. 
	
	More precisely, let $\beta$ be the unique geodesic that minimizes the distance between $z_1$ and $\equator_2$.  Extend the geodesic $\beta$ to meet the equator $\equator_1$, and denote this extended geodesic still by $\beta$. Re-parameterize\footnote{The curve $\beta$ may not have unit speed any longer after such a re-parameterization. But this will not affect our discussion.} $\beta$ so that its domain becomes $[1, 2]$, and  $\beta(1) \in \equator_1$ and $\beta(2) \in \equator_2$.   For every $t\in [1, 2]$, let $v_t$ be the tangent vector of $\beta$ at $\beta(t)$. Let $\equator_{t}$ be the unique equator that is orthogonal to $v_t$.   
	
	Let $P_{z_1}$ be the set of $\omega_1$-positive points in $\Phi_{\equator_1}(\Sigma)$ that are contained in $\sph_{\equator_2, -}$. Define a level map $L\colon P_{z_1} \to [1, 2]$ by setting $L(x) = t $ if $x$ is contained in $\equator_t$. Denote the values in $L(P_{z_1})$ by $\{t_j\}_{0\leq j \leq N}$ with $t_0\leq t_1 \leq t_2\leq \cdots$. Fix a number $\delta >0$ that is very small compared to the differences $(t_{j+1} - t_{j})$ for all $0\leq j\leq N$. Let us denote by $P_{z_1}(t_j) = P_{z_1}\cap E_{t_j}$ the intersection of $P_{z_1}$ and $E_{t_j}$. 
	
	Let $s_1 = t_0 + \frac{t_1 - t_0}{3}$ and $\Phi_{s_1}$ a wrapping map along the equator $\equator_{s_1}$ such that $\Phi_{s_1}$ equals the reflection map along $\equator_{s_1}$ on small neighborhoods of elements $x\in P_{z_1}(t_0)$ and   $\Phi_{s_1}$ equals the identity map on small neighborhoods of elements $x\in P_{z_1}(t_j)$ for $j\geq 1$. In particular, for $x\in P_{z_1}(t_0)$, its image $\Phi_{s_1}(x)$ under the map $\Phi_{s_1}$ lies in $\equator_{r_1}$, where $r_1 = t_0 + \frac{2(t_1 - t_0)}{3}$. Now set\footnote{Here the positive number $\delta$ is added to make sure the wrapping map $\Phi_{s_2}$ remains injective on $\Sigma$. } 
	\[ s_2 =  t_0 + \frac{2(t_1 - t_0)}{3} +  \frac{t_1 - t_0}{6}+ \delta.\] Let $\Phi_{s_2}$ be a wrapping map along the equator $\equator_{s_2}$ such that $\Phi_{s_2}$ equals the reflection map along $\equator_{s_2}$ on small neighborhoods of elements $x\in \Phi_{s_1}(P_{z_1}(t_0))$ and   $\Phi_{s_2}$ equals the identity map on small neighborhoods of elements  $x\in \Phi_{s_1}(P_{z_1}(t_j))$ for $j\geq 1$. In particular, since the composition of any two reflections is an element of $\SO(n+1)$, it follows that the composition $\Phi_{s_2}\circ \Phi_{s_1}$ equals an element of $\SO(n+1)$ on small neighborhoods of elements $x\in P_{z_1}(t_0)$, and equals the identity map on small neighborhoods of elements  $x\in P_{z_1}(t_j))$ for $j\geq 1$. Furthermore, the levels of  ${\Phi_{s_2}\circ \Phi_{s_1} (P_{z_1}(t_0))}$ and $\Phi_{s_2}\circ \Phi_{s_1} (P_{z_1}(t_1))$ are very close. More precisely,  
	\[  L(x_0) - L(x_1) = 2\delta \]
	for all $x_0\in \Phi_{s_2}\circ \Phi_{s_1} (P_{z_1}(t_0))$ and $x_1\in \Phi_{s_2}\circ \Phi_{s_1} (P_{z_1}(t_1)) = P_{z_1}(t_1)$. We shall call the composition $\Phi_{s_2}\circ \Phi_{s_1}$ a doubling wrapping map. Roughly speaking, a double wrapping map brings the points of $P_{z_1}$ closer to $z_1$, while preserving the orientation at those points. Now it is not difficult to see that there is a finite sequence of double wrapping maps $\Phi_{s_1}, \cdots, \Phi_{s_k}$ such that 
	\begin{enumerate}
		\item on a small neighborhood of each point in $P_{z_1}$, the composition 
		\[ {\widetilde \Phi  \coloneqq \Phi_{s_k}\circ \cdots \circ \Phi_{s_1}} \textup{ equals an element of $\SO(n+1)$},\]
		\item  $\widetilde \Phi$ equals the identity map on a small neighborhood of $z_1$,  
		\item  $\widetilde \Phi$ moves all points in $P_{z_1}$ past $z_1$, that is, 
		\[ L(\widetilde \Phi(x)) > L(z_1) \]
		for all $x\in P_{z_1}$, where $L$ is the level map from above. 
	\end{enumerate}
	Therefore, we are reduced to the case where  $\sph_{\equator_2,-}$ contains no $\omega_1$-positive points of $\Phi_{\equator_1}(\Sigma)$.  In this case,  we define  $\Phi_{\equator_2}$ to be a wrapping map along $\equator_2$  such that $\Phi_{\equator_2}$ equals the reflection map along $\equator_2$ on a small neighborhood of $z_1$, and $\Phi_{\equator_2}$ equals the identity map on small neighborhoods of the remaining points of $\Phi_{\equator_1}(\Sigma) \backslash\{z_1\}$. Let  
	\[ \Phi_{2} = \Phi_{\equator_2}\circ \Phi_{\equator_1}\] be the composition of $\Phi_{\equator_2}$ and $\Phi_{\equator_1}$. We define its associated orientation-indicator function $\omega_{2}\colon \Sigma \to \{\pm1\}$ by setting 
	\[ \omega_{2}(x) = \begin{cases}
		+1 & \textup{ if $\Phi_{2}$ is orientation-preserving on a small neighborhood of $x$,} \\
		-1 & \textup{ if $\Phi_{2}$ is orientation-reserving on a small neighborhood of $x$.}
	\end{cases} \]
	In particular,  we have 
	\[  \sum_{x\in \Sigma} \omega_2(x)  = \sum_{x\in \Sigma} \omega_1(x) + 2 > \sum_{x\in \Sigma} \omega_1(x). \]
	In other words, the total number of $\omega$-positive points is strictly increasing.  Since $\Sigma$ is a finite set, every point of $\Sigma$ will eventually become $\omega$-positive within   finitely many steps. Then the composition of all the wrapping maps appearing in these inductive steps gives the map $\Phi\colon \sph^n \to \sph^n$ with the desired properties.  This finishes the proof. 	
\end{proof}

\end{document}